\documentclass[a4paper]{article}

\usepackage{geometry}
\usepackage{graphicx}
\usepackage{parskip}

\usepackage{amssymb}
\usepackage{amsthm}
\usepackage{amsmath}

\usepackage{listings}
\usepackage{enumerate}

\usepackage[english]{babel}
\usepackage[utf8]{inputenc}

\usepackage[pdftex, bookmarks=true, colorlinks=true, linkcolor=black, urlcolor=black, citecolor=black, linktoc=section]{hyperref}

\usepackage[multiple, symbol]{footmisc}

\usepackage[title]{appendix}

\newtheorem{theorem}{Theorem}

\newtheorem{lemma}[theorem]{Lemma}
\newtheorem{proposition}[theorem]{Proposition}
\newtheorem{remark}[theorem]{Remark}
\newtheorem{definition}[theorem]{Definition}
\newtheorem{example}[theorem]{Example}

\numberwithin{theorem}{section}
\numberwithin{equation}{section}

\newcommand{\mR}{\mathbb{R}}
\newcommand{\mS}{\mathbb{S}}
\newcommand{\mC}{\mathbb{C}}
\newcommand{\mD}{\mathbb{W}}
\newcommand{\mM}{{\mathbb M}}
\newcommand{\gln}{{\mathbb{G}\mathbb{L}_n}}
\newcommand{\mU}{{\mathbb U}}
\newcommand{\mH}{{\mathbb H}}
\newcommand{\mP}{{\mathbb P}}
\newcommand{\mA}{{\mathbb A}}
\newcommand{\scriptU}{\mD\mU_n}
\newcommand{\scriptUA}{\mA\mU_n}

\newcommand{\diag}{{\text{\rm diag}}}
\newcommand{\tr}{{\text{\rm tr}}}
\newcommand{\logrank}{{\text{\rm log-rank}}}

\let\odlangle\angle  
\renewcommand{\angle}{\odlangle \,}

\begin{document}


\title{Finsler geometries on strictly accretive matrices}
\author{}
\date{\vspace{-40pt}
Axel \textsc{Ringh}\footnotemark[1] \; and \ Li \textsc{Qiu}\footnotemark[1]
}

\maketitle

\begin{abstract}
In this work we study the set of strictly accretive matrices, that is, the set of matrices with positive definite Hermitian part, and show that the set can be interpreted as a smooth manifold.
Using the recently proposed symmetric polar decomposition for sectorial matrices, we show that this manifold is diffeomorphic to a direct product of the manifold of (Hermitian) positive definite matrices and the manifold of strictly accretive unitary matrices. Utilizing this decomposition, we introduce a family of Finsler metrics on the manifold and charaterize their geodesics and geodesic distance. Finally, we apply the geodesic distance to a matrix approximation problem, and also give some comments on the relation between the introduced geometry and the geometric mean for strictly accretive matrices as defined
by S.~Drury in [S.~Drury, Linear Multilinear Algebra. 2015 63(2):296--301].
\end{abstract}


\renewcommand{\thefootnote}{\fnsymbol{footnote}}
\footnotetext[1]{Department of Electronic and Computer Engineering,\ The Hong Kong University of Science and Technology,\ Clear Water Bay, Kowloon, Hong Kong, China.
Email: \texttt{eeringh@ust.hk} (A.~Ringh), \texttt{eeqiu@ust.hk} (L.~Qiu)}
\footnotetext[2]{This work was supported by the Knut and Alice Wallenberg foundation, Stockholm, Sweden, under grant KAW 2018.0349, and the Hong Kong Research Grants Council, Hong Kong, China, under project GRF 16200619.}

\renewcommand{\thefootnote}{\arabic{footnote}}
\setcounter{footnote}{0}


\section{Introduction}
Given a complex number $z \in \mC$ we can write it in its Cartesian form $z = a +ib$, where $a = \Re(z)$ is the real part and $b = \Im(z)$ is the imaginary part, or we can write it in its polar form as $z = r e^{i\theta}$, where $r = |z|$ is the magnitude and $\theta = \angle z$ is the phase. The standard metric on $\mC$ defines the (absolute) distance between $z_1$ and $z_2$ as $|z_1 - z_2| = \sqrt{\Re(z_1 - z_2)^2 + \Im(z_1 - z_2)^2}$ which is efficiently computed using the Cartesian form as $\sqrt{(a_1 - a_2)^2 + (b_1 - b_2)^2}$. However, sometimes a logarithmic (relative) distance between the numbers contains information that is more relevant for the problem at hand. One such distance is given by  $\sqrt{\log(r_1/r_2)^2 + [(\theta_1 - \theta_2) \mid \text{mod } 2\pi]^2}$, and in this distance measure the point $1$ is as close to $10e^{i\theta}$ as it is to $0.1e^{-i\theta}$. 
This type of distances have wide application in engineering problems, e.g., as demonstrated in the use of  Bode plots and Nichols charts in control theory \cite{astrom2008feedback}.
Moreover, this type of logarithmic metric has been generalized to (Hermitian) positive definite matrices, with plenty of applications, for example in computing geometric means between such matrices \cite{moakher2005differential}, \cite[Chp.~6]{bhatia2007positive}, \cite[Chp.~XII]{lang1999fundamentals}. This generalization can done by identifying the set of positive matrices as a smooth manifold and introducing a Riemannian or Finsler metric on it. Here, we follow a similar path and extend this type of logarithmic metrics to so called strictly accretive matrices.
More specifically,
the outline of the paper is as follows: in Section~\ref{sec:background} we review relevant background material and set up the notation used in the paper. Section~\ref{sec:A_smooth_manifold} is devoted to showing that the set of strictly accretive matrices can be interpreted as a smooth manifold, and that this manifold is diffeomorphic to a direct product of the smooth manifold of positive definite matrices and the smooth manifold of strictly accretive unitary matrices.
The latter is done using the newly introduced symmetric polar decomposition for sectorial matrices \cite{wang2019phases}.
In Section~\ref{sec:riemann_A} we introduce a family of Finsler metrics on the manifold, by means of the decomposition from the previous section and so called (Minkowskian) product functions \cite{okada1982minkowskian}. In particular,
this allows us to characterize the
corresponding geodesics and compute the geodesic distance. Finally, in Section~\ref{sec:applications} we given an application of the metric to a matrix approximation problem and also give some comments on the relation between the geodesic midpoint and the geometric mean between strictly accretive matrices as introduced in \cite{drury2015principal}.

\section{Background and notation}\label{sec:background}
In the following section we introduce some background material needed for the rest of the paper. At the same time, this section is also be used to set up the notation used throughout.
To this end, let $\mM_n$ denote the set of $n \times n$ matrices over the filed $\mC$ of complex numbers.
For $A \in \mM_n$, let $A^*$ denotes its complex conjugate transpose, let $H(A) := \tfrac{1}{2} (A + A^*)$ denote its Hermitian part, and let $S(A) := \tfrac{1}{2} (A - A^*)$ denote its skew-Hermitian part.%
\footnote{Note that this is equivalent to the Toeplitz decomposition since if $A = \Re(A) + i \Im(A)$, then $\Re(A) = H(A)$ and $\Im(A) = \tfrac{1}{i} S(A)$, see \cite[p.~7]{horn2013matrix}.}
Moreover, by $I$ we denote the identity matrix, and for $A \in \mM_n$ by $\lambda(A)$ we denote its spectrum, i.e., $\lambda(A) := \{ \lambda \in \mC \mid \det(\lambda I - A) = 0 \}$, and by $\sigma(A)$ we denote it singular values, i.e., $\sigma(A) = \sqrt{\lambda(A^*A)}$.

The following is a number of different sets of matrices that will be used throughout: 
$\gln$ denotes the set of invertible matrices,
$\mU_n$ denotes the set of unitary matrices,
$\mH_n$ denotes the set of Hermitian matrices,
$\mP_n$ denotes the set of positive definite matrices, i.e., $A \in \mH_n$ s.t. $\lambda(A) \subset \mR_+ \setminus \{ 0 \}$,
$\mS_n$ denotes the set of skew-Hermitian matrices,
and
$\mA_n$ denotes the set of strictly accretive matrices, i.e., $A \in \mA_n$ if and only if $H(A) \in \mP_n$.%
\footnote{The naming used here is the same as in \cite[p.~281]{kato1995perturbation}, in contrast to \cite{ballantine1975accretive}.}

Two matrices $A,B \in \mM_n$ are said to be congruent if there exists a matrix $C \in \gln$ such that $A = C^*BC$.
For matrices $A,B \in \mM_n$ we define the inner product $\langle A, B \rangle := \tr(A^*B)$, which gives the Frobenius norm $\| A \| := \sqrt{\langle A, A \rangle} = \sqrt{\sum_{j=1}^n \sigma_j(A)^2}$.
By $\| \cdot \|_{\text{sp}}$ we denote the spectral norm, i.e., $\| A \|_{\text{sp}} = \sup_{x \in \mC^n \setminus \{0\}} \| Ax \|_2/ \| x \|_2 = \sigma_{\max}(A)$, the larges singular value of $A$. Next, a function $\Phi : \mR^n \to \mR$ is called a symmetric gauge function if for all $x,y \in \mR^n$ and all $\beta \in \mR$ i)~$\Phi(x) > 0$ if $x \neq 0$, ii) $\Phi(\beta x) = |\beta| \Phi(x)$, iii) $\Phi(x + y) \leq \Phi(x) + \Phi(y)$, and iv) $\Phi(x) = \Phi(\tilde{x})$ for all $\tilde{x} = [\pm x_{\alpha(i)}]_{i = 1}^n$ where $\alpha$ is any permutation of $\{ 1, \ldots, n \}$ \cite{mirsky1960symmetric}, \cite[Sec.~3.I.1]{marshall2011inequalities}. For any unitary invariant norm, i.e., norms $|\| \cdot \||$ such that $|\| U A V \|| = |\| A \||$ for all $A \in \mM_n$ and all $U,V \in \mU_n$, there exists a symmetric gauge function $\Phi$ such that $|\| A \|| = \Phi(\sigma(A))$  \cite[Thm~IV.2.1 ]{bhatia1997matrix}, \cite[Thm.~10.A.1]{marshall2011inequalities}. For this reason we will henceforth denote such norms $\|  \cdot \|_{\Phi}$. Moreover, we will call a symmetric gauge function, and the corresponding norm, smooth if it is smooth outside of the origin, cf.~\cite[Thm.~8.5]{lewis1996group}.

For a vector $x \in \mR^n$, by $x^{\downarrow}$ we denote the vector obtained by sorting the elements in $x$ in a nonincreasing order. More precisely, $x^{\downarrow}$ is obtained by permuting the elements of $x$ such that $x^{\downarrow} = [x_k^\downarrow]_{k=1}^n$  where $x_1^\downarrow \geq x_2^\downarrow \geq \ldots \geq x_n^{\downarrow}$.
For two vectors $x,y \in \mR^n$, we say that $x$ is submajorized (weakly submajorized) by $y$ if $\sum_{k = 1}^{\ell} x_k^\downarrow \leq \sum_{k = 1}^{\ell} y_k^\downarrow$ for $\ell = 1, \ldots, n-1$ and  $\sum_{k = 1}^{n} x_k^\downarrow = (\leq) \sum_{k = 1}^{n} y_k^\downarrow$ \cite[p.~12]{marshall2011inequalities}. Submajorization (weak submajorization) is a preorder on $\mR^n$, and we write $x \prec (\prec_w) y$. On the equivalence classes of vectors sorted in nonincreasing order it is a partial ordering  \cite[p.~19]{marshall2011inequalities}.

\subsection{Sectorial matrices and the phases of a matrix}

Given a matrix $A \in \mM_n$, we define the numerical range (field of values) as
\[
W(A) := \big\{ z \in \mC \mid z = x^*Ax, \; x \in \mC^n \text{ and } \| x \|^2 = x^*x = 1 \big\}.
\]
Using the numerical range, we can define the set of so called \emph{sectorial matrices} as 
\[
\mD_n := \{ A \in \mM_n \mid 0 \not \in W(A) \}.
\]
The name comes from the fact that the numerical range of a matrix $A \in \mD_n$ is contained in a sector of opening angle less than $\pi$. The latter can be seen from the Toeplitz-Hausdorff theorem, which states that for any matrix $A \in \mM_n$, $W(A)$ is a convex set \cite[Thm.~4.1]{zhang2011matrix}, \cite[Thm.~1.1-2]{gustafson1997numerical}. Recently, sectorial matrices have received considerable attention in the literature, see, e.g., \cite{ballantine1975accretive, mathias1992matrices, drury2014singular, li2014determinantal, zhang2015matrix, wang2019phases, chen2019phase}.

Sectorial matrices have several interesting properties. In particular,
if $A$ is sectorial it is congruent to a unitary diagonal matrix $D$, i.e.,
$
A = T^*DT
$
for some $T \in \gln$ \cite{horn1959eigenvalues, deprima1974range, furtado2001spectral, johnson2001generalization, horn2006canonical}. Although the decomposition is not unique, the elements in $D$ are unique up to permutation, and any such decomposition is called a \emph{sectorial decoposition} \cite{zhang2015matrix}. Using this decomposition, we define the phases of $A$, denoted $\phi_1(A), \phi_2(A), \ldots, \phi_n(A)$, as the phases of the eigenvalues of $D$ \cite{furtado2001spectral, wang2019phases, chen2019phase};%
\footnote{In \cite{furtado2001spectral}, these were called canonical angles.}
by convention we defined them to belong the an interval of length strictly less than $\pi$.
With this definition we have, e.g., that
$A \in \mA_n$ if and only if $A \in \mD_n$ and $H(A) \in \mP_n$, which is true if and only if $(\phi_1(A), \ldots, \phi_n(A)) \subset (-\pi/2, \pi/2)$.
Note that the phases of a sectorial matrix $A$ is different from the angles of the eigenvalues, i.e., in general $\phi(A) \neq \varphi(A)$ where $\varphi(A) := \angle \lambda(A)$.
More precisely, equality holds for normal matrices.
The phases have a number of desirable properties that the angles of the eigenvalues do not, see \cite{wang2019phases}.

Another decomposition of sectorial matrices, which will in fact be central to this work, is the so called \emph{symmetric polar decomposition} \cite[Thm.~3.1]{wang2019phases}: for $A \in \mD_n$ there is a unique decomposition given by 
\[
A = P U P,
\]
where $P \in \mP_n$ and $U \in \scriptU :=  \mU_n \cap \mD_n$. The latter is the set of sectorial unitary matrices. The phases of $A$ are given by the phases of $U$, which is in fact the phases of the eigenvalues of $U$. Therefore, we have that $A \in \mA_n$ if and only if it has a symmetric polar decomposition such that $U \in \scriptUA := \{ U \in \scriptU \mid H(U) \in \mP_n \}$, i.e., the set of strictly accretive unitary matrices.

\subsection{Riemannian and Finsler manifolds}
Smooth manifolds are important mathematical objects which show up in such diverse fields as theoretical physics \cite{oneill1983semi}, robotics \cite{murray1994mathematical}, and statistics and information theory \cite{amari2000methods}. Intuitively, they can be thought of as spaces that locally look like the Euclidean space, and on these spaces one can introduce geometric concepts such as curves and metrics. 
In particular, all smooth manifolds admit a so called Riemannian metric \cite[Thm.~1.4.1]{jost2008riemannian}, \cite[Prop.~13.3]{lee2013introduction}, and 
Riemannian geometry is a well-studied subjects, see, e.g., one of the monographs \cite{oneill1983semi, lang1999fundamentals, jost2008riemannian, lee2013introduction, lee2018introduction}.
An relaxation of Riemannian geometry leads to so called Finsler geometry \cite{chern1996finsler}; loosely expressed it can be interpreted as chaining the tangent space from a Hilbert space to a Banach space. For an introduction to Finsler geometry, see, e.g., \cite{bao2000introduction, shimada2000finsler, cheng2012finsler}. 

More specifically, given a smooth manifold $\mathcal{M}$, for $x \in \mathcal{M}$ we denote the tangent space by $T_x\mathcal{M}$ and the tangent bundle by $T\mathcal{M} := \cup_{x \in \mathcal{M}} \{x\} \times T_x\mathcal{M}$.
A Riemannian metric is induced by an inner product on the tangent space, $\langle \cdot, \cdot \rangle_x : T_x\mathcal{M} \times T_x\mathcal{M} \to \mR$, that varies smoothly with the base point $x$. Using this inner product, one defines the norm $\sqrt{\langle \cdot, \cdot \rangle_x}$, which in fact defines a smooth function on the slit tangent bundle $T\mathcal{M} \setminus \cup_{x \in \mathcal{M}} (x, 0)$.
In this work we consider Finsler structures on smooth (matrix) manifolds, but we will limit the scope to Finsler structures $F : T\mathcal{M} \to \mR_+$, $ (x,X)  \mapsto \| X \|_x$, where $\|  \cdot \|_x$ is a norm on $T_x\mathcal{X}$ which is not necessarily induced by an inner product, and such that $F$ is smooth on the slit tangent bundle $T\mathcal{M} \setminus \cup_{x \in \mathcal{M}} (x, 0)$.
Given a piece-wise smooth curve $\gamma : [0,1] \to \mathcal{M}$, the arc length on the manifold is defined using this Finsler structure. More precisely, it is defined as
\begin{equation*}
\mathcal{L}(\gamma) := \int_0^1 F(\gamma(t), \dot{\gamma}(t)) dt,
\end{equation*}
where $\dot{\gamma}(t)$ is the derivative of $\gamma$ with respect to $t$. Using arc length, the geodesic distance between two points $x,y \in \mathcal{M}$ is defined as
\begin{equation*}
\delta (x, y) := \inf_{\gamma}  \mathcal{L}(\gamma) \; : \; \gamma \text{ is a piece-wise smooth curve such that } \gamma(0)=x, \gamma(1)=y,
\end{equation*}
and a minimizing curve (if one exists) is called a geodesic.
A final notion we need is that of diffeomorphic manifolds. More precisely, two smooth manifolds $\mathcal{M}$ and $\mathcal{N}$ are said to be diffeomorphic if there exists a diffeomorphism $f : \mathcal{M} \to \mathcal{N}$, i.e., a function $f$ which is a smooth bijection with a smooth inverse. In this case we write $\mathcal{M} \cong \mathcal{N}$.

Next, we summarize some results regarding two matrix manifolds, namely $\mP_n$ and $\mU_n$, together with specific Finsler structures. These will be needed later.

\subsubsection{A family of Finsler metrics on $\mP_n$ and their geodesics}
Riemannian and Finsler geometry on $\mP_n$ is a well-studied subject, and we refer the reader to, e.g., \cite{moakher2005differential}, \cite[Chp.~6]{bhatia2007positive} or \cite[Chp.~XII]{lang1999fundamentals}. Here, we summarize some of the results we will need for later. To this end, note that the tangent space at $P \in \mP_n$ is $\mH_n$, and given $P \in \mP_n$, $X,Y \in \mH_n$ we can introduce the inner product on the tangent space as 
\begin{equation}\label{eq:inner_product_P}
\langle X, Y \rangle_{P} = \tr((P^{-1/2} X P^{-1/2})^*(P^{-1/2} Y P^{-1/2})),
\end{equation}
with corresponding norm $F_{\mP_n}(P, X) = \| X \|_P
= \sqrt{\sum_{j=1}^n \sigma_j(P^{-1/2}XP^{-1/2})^2}$.
The geodesic between $P,Q \in \mP_n$ in the induced Riemannian metric is given by
\begin{equation}\label{eq:geodesic_on_P}
\gamma_{\mP_n}(t) = P^{1/2} (P^{-1/2} Q P^{-1/2})^t P^{1/2} = Pe^{t \log(P^{-1}Q)}
\end{equation}
and the length of the curve, i.e., the Riemannian distance between $P$ and $Q$, is given by 
\[
\delta_{\mP_n}(P, Q) = \| \log (P^{-1/2} Q P^{-1/2}) \|.
\]
Interestingly, if the norm $\| \cdot \|_P$ on $T_P\mP_n$ is changed to any other unitary invariant matrix norm
\[
\| \cdot \|_{\Phi, P} = \Phi(\sigma(P^{-1/2} \cdot P^{-1/2}))
\]
the expressions for a geodesic between two matrices remains unchanged and the corresponding distance is given by $\delta_{\mP_n}^{\Phi}(P, Q) = \| \log (P^{-1/2} Q P^{-1/2}) \|_{\Phi}$ \cite[Sec.~6.4]{bhatia2007positive}. However, the geodesic \eqref{eq:geodesic_on_P} might no longer be the unique shortest curve \cite[p.~223]{bhatia2007positive}.

An alternative expression for the geodesic \eqref{eq:geodesic_on_P} is given by the following proposition.

\begin{proposition}\label{prop:new_form_curve_P}
Let $\| \cdot \|_{\Phi}$ be any smooth unitarily invariant norm and consider the Finsler structure given by $F_{\mP_n}^{\Phi}: T \mP_n \to \mR_+$, $F_{\mP_n}^{\Phi} : (P, X) \mapsto \| P^{-1/2} X P^{-1/2} \|_{\Phi}$.
For $P,Q \in \mP_n$, a geodesic between them can be written as $\gamma(t) = S \Lambda^t S^*$ where  $P = SS^*$, $Q = S \Lambda S^*$ is a simultaneous diagonalization by congruence of $P$ and $Q$, i.e., $S \in \gln$ and $\Lambda$ is diagonal with positive elements on the diagonal. Moreover, the geodesic distance is
\begin{equation}\label{eq:dist_P}
\delta_{\mP_n}^{\Phi}(P, Q) = \| \log (P^{-1/2} Q P^{-1/2}) \|_{\Phi} = \Phi\Big( \log \big( \lambda(P^{-1} Q) \big) \Big) = \| \log (\Lambda) \|_{\Phi}.
\end{equation}
\end{proposition}

\begin{proof}
To show the second equality in \eqref{eq:dist_P}, note that
\begin{align*}
\| \log (P^{-1/2} Q P^{-1/2}) \|_{\Phi} & = \Phi\Big( \sigma \big( \log(P^{-1/2} Q P^{-1/2}) \big) \Big) = \Phi\Big( |\lambda \big( \log(P^{-1/2} Q P^{-1/2}) \big)| \Big) \\
& = \Phi\Big( \lambda \big( \log(P^{-1/2} Q P^{-1/2}) \big) \Big) = \Phi\Big( \log \big(  \lambda (P^{-1/2} Q P^{-1/2}) \big) \Big) \\
& = \Phi\Big( \log \big(  \lambda (P^{-1} Q) \big) \Big)
\end{align*}
where the second equality comes from that the singular values of a Hermitian matrix, i.e., the matrix $\log(P^{-1/2} Q P^{-1/2})$, are the absolute values of the eigenvalues, the third equality follows since the symmetric gauge function is invariant under sign changes, the forth can be seen by using a unitary diagonalization of $P^{-1/2} Q P^{-1/2}$, and the fifth equality comes from that the spectrum is invariant under similarity.
Next, since both $P, Q \in \mP_n$, by \cite[Thm.~7.6.4]{horn2013matrix} we can simultaneously diagonalize $P$ and $Q$ by congruence, i.e., there exists an $S \in \gln$ such that $P = SS^*$ and $Q = S\Lambda S^*$, where $\Lambda=\diag(\lambda_1, \ldots, \lambda_n)$ and where $\lambda_1, \ldots, \lambda_n$ are all strictly larger than $0$. In fact, $\lambda_1, \ldots, \lambda_n$ are the eigenvalues of $P^{-1}Q$, which means that 
$
\log \big( \lambda(P^{-1} Q) \big) = \log (\lambda(\Lambda)) = \log (\Lambda),
$
which in turn gives the last equality in \eqref{eq:dist_P}.
Finally, this also gives that
$
\gamma(t) = Pe^{t \log(P^{-1}Q)} = SS^*  S^{-*} e^{t\log (\Lambda)} S^*= S \Lambda^t S^*.
$
\end{proof}

\subsubsection{A family of Finsler metrics on $\mU_n$ and their geodesics}
The set of unitary matrices is a Lie group, and results related to Riemannian and Finsler geometry on $\mU_n$ can be found in, e.g., \cite{andruchow2005short, andruchow2008geometry, andruchow2010finsler, antezana2014optimal, antezana2020minimal}. 
Again, we here summarize some of the results that we will need for later. To this end, note that the tangent space at $U \in \mU_n$ is $\mS_n$ and given $U \in \mU_n$, $X,Y \in \mS_n$  we can introduce the inner product on the tangent space as 
\begin{equation}\label{eq:inner_product_U}
\langle X,Y \rangle_{U} = \tr(X^*Y),
\end{equation}
with corresponding induced norm $F_{\mU_n}(U, X)  = \| X \|_U = \sqrt{\langle X, X \rangle_U}$.
The induced Riemannian metric have shortest curves between $U,V \in \mU_n$ given by
\begin{equation}\label{eq:geodesic_on_U}
\gamma_{\mU_n}(t) = Ue^{itZ},
\end{equation}
where $V = Ue^{iZ}$, and where $Z \in \mH_n$ is such that $\| Z \|_{\text{sp}} \leq \pi$. 
Moreover, the geodesic distance is
\begin{equation}\label{eq:geodesic_distance_U}
\delta_{\mU_n}(U, V) =  \| Z \|,
\end{equation}
and the geodesic is unique if $\| Z \|_{\text{sp}} < \pi$. Similarly to the results for the smooth manifold $\mP_n$, if the norm $\| \cdot \|_U$ on $T_U \mU_n$ is changed to any other unitary invariant matrix norm $\| \cdot \|_{\Phi, U}$
the expressions for a geodesic in \eqref{eq:geodesic_on_U} is unchanged and
the expression for the geodesic distance \eqref{eq:geodesic_distance_U} is only changed to using the corresponding norm $\| \cdot \|_{\Phi}$ \cite{andruchow2010finsler, antezana2014optimal}.
However, even if $\| Z \|_{\text{sp}} < \pi$ the geodesic \eqref{eq:geodesic_on_U} might not be unique in this case \cite[Sec.~3.2]{antezana2014optimal}.

\section{The smooth manifold of strictly accretive matrices}\label{sec:A_smooth_manifold}
In this section we prove that $\mA_n$ is a smooth manifold diffemorphic to $\mP_n \times \scriptUA$.
The latter is fundamental for the introduction of the Finsler structures in the following section.
For improved readability, some of the technical results of this section are deferred to Appendix~\ref{app:technical_results_one}. To this end, we start by proving the following.

\begin{theorem}\label{thm:A_smooth_manifold}
$\mA_n$ is a connected smooth manifold, and at a point $A \in \mA_n$ the tangent space is $T_A\mA_n = \mM_n$.
\end{theorem}

\begin{proof}
This follows by Lemma~\ref{lem:open_A}, \ref{lem:connected}, and \ref{lem:tangent_space}, and applying \cite[Ex.~1.26]{lee2013introduction}.
\end{proof}

Next, we prove the following characterization of the manifold.

\begin{theorem}\label{thm:decomp_manifold}
$\mA_n \cong \mP_n \times \scriptUA$, where $\mP_n$ and $\scriptUA$ are embedded submanifolds.
\end{theorem}

This theorem follows as a corollary to the following proposition.

\begin{proposition}\label{prop:smooth_varying_SPD}
For $A \in \mA_n$, let $A = PUP$ be the symmetric polar decomposition.
Then the mapping $A \mapsto (P^2, U)$ is a diffeomorphism between the smooth manifolds $\mA_n$ and $\mP_n \times \scriptUA$.
\end{proposition}

\begin{proof}
Since $\mP_n$ and $\scriptUA$ are smooth manifolds (see \cite[Chp.~6]{bhatia2007positive} and Lemma~\ref{lem:script_U_manif}, respectively), $\mP_n \times \scriptUA$ is also a smooth manifold \cite[Ex.~1.34]{lee2013introduction}. 
Next, note that the matrix square root is a diffeomorphism of $\mP_n$ to itself, with the matrix square as the inverse, cf.~\cite[Thm.~7.2.6]{horn2013matrix}. Therefore, it suffices to show that the mapping $A \mapsto (P, U)$ is a diffeomorphism between $\mA_n$ and $\mP_n \times \scriptUA$.
To this end, first observe that the latter 
is a bijection due to
the existence and uniqueness of a symmetric polar decomposition \cite[Thm.~3.1]{wang2019phases}.
Moreover, that the inverse is smooth follows since the components in $A$ are polynomial in the components in $P$ and $U$. 

To show that $P$ and $U$ are smooth in $A$, 
we note that since $A$ is strictly accretive, $H(A) \succ 0$. This means that we can write
\begin{align*}
A & = H(A) + i (\tfrac{1}{i}S(A)) = H(A)^{1/2}\big(I + i H(A)^{-1/2}\tfrac{1}{i}S(A)H(A)^{-1/2}\big)H(A)^{1/2} \\
& = H(A)^{1/2} K H(A)^{1/2},
\end{align*}
where $K := I + iH(A)^{-1/2}\tfrac{1}{i}S(A)H(A)^{-1/2}$ is a normal matrix%
\footnote{To see this, note that a matrix $A$ is normal if and only if $H(A)$ and $\tfrac{1}{i}S(A)$ commute, see, e.g., \cite[Thm.~9.1]{zhang2011matrix}.}
(cf.~\cite[Proof of Cor.~2.5]{zhang2015matrix})
which by construction depends smoothly on $A$.
Now, let $K = V_KQ_K$ be the polar decomposition of $K$. Since $K$ depends smoothly on $A$, and since the polar decomposition is smooth in the matrix (Lemma~\ref{lem:polar_decomp_smooth}), $V_K$ and $Q_K$ are smooth in $A$. Moreover, since $K$ is normal, $V_K$ and $Q_K$ commute \cite[Thm.~9.1]{zhang2011matrix}, and thus $V_K$ and $Q_K^{1/2}$ commute. Therefore, $A = H(A)^{1/2}Q_K^{1/2}U_KQ_K^{1/2}H(A)^{1/2}$,
where all components depend smoothly on $A$. Now, let $L := Q_K^{1/2}H(A)^{1/2}$ and let $L = V_L Q_L$ be the polar decomposition of $L$. Similar to before, $V_L$ and $Q_L$ are thus both smooth in $A$. Finally, we thus have that
$A = Q_L V_L^*U_K V_L Q_L,$
and since $V_L$ and $U_K$ are unitary so is $V_L^*U_KV_L$. By the uniqueness of the symmetric polar decomposition \cite[Thm.~3.1]{wang2019phases}
it follows that
$P = Q_L$ and $U = V_L^*U_K V_L$,
which are both smooth in $A$. 
\end{proof}

\begin{proof}[Proof of Theorem~\ref{thm:decomp_manifold}]
By Proposition~\ref{prop:smooth_varying_SPD}, $\mA_n \cong \mP_n \times \scriptUA$, and by \cite[Prop.~5.3]{lee2013introduction}, both $\mP_n \times \{ I \} \cong \mP_n$ and $\{ I \} \times \scriptUA \cong \scriptUA$ are embedded submanifolds of $\mP_n \times \scriptUA$.
\end{proof}

\begin{remark}\label{rem:rotation_of_A}
The results in Theorem~\ref{thm:A_smooth_manifold} and \ref{thm:decomp_manifold} can easily be generalized to other subsets of sectorial matrices, namely any subset $\tilde{\mA}_n \subset \mD_n$ of all matrices $A$ such that there exists $\alpha, \beta \in \mR$, $\alpha < \beta$, and $\beta - \alpha = \pi$, for which $\min_{k = 1, \ldots n} \phi_k(A) > \alpha$ and $\max_{k = 1, \ldots n} \phi_k(A) < \beta$. To see this, note that for $\tilde{A} = \tilde{P}\tilde{U}\tilde{P} \in \tilde{\mA}_n$ we have that  $A = \tilde{P} (e^{-(\beta + \alpha)/2}\tilde{U}) \tilde{P} \in \mA_n$, i.e., that $ e^{-(\beta + \alpha)/2} \tilde{\mA}_n =\mA_n$ with a diffeomorphism between the components in the symmetric polar decomposition. Examples of such sets of matrices are 
and the set of strictly dissipative matrices, i.e., matrices such that $\Re(A) \prec 0$ \cite[p.~279]{kato1995perturbation}, and matrices such that $\Im(A) \succ 0$.%
\footnote{Note that the latter has unfortunately also been termed dissipative in the literature, see \cite[p.~279]{kato1995perturbation}.}
\end{remark}

\section{A family of Finsler metrics on $\mA_n$ and their geodesics}\label{sec:riemann_A}

In this section we introduce a family of Finsler structures on $\mA_n$
and in particular we will characterize the geodesics and geodesic distances corresponding to these structures.
To this end,  by Theorem~\ref{thm:decomp_manifold} we have that $\mA_n \cong \mP_n \times \scriptUA$.
Moreover, $\mP_n$ and $\mU_n$ are smooth manifolds that are well-studied in the literature, and since $\mP_n$ and $\scriptUA$ are embedded submanifolds of $\mA_n$ a desired property
would be that when restricted to any of the two embedded submanifold the introduced Finsler structure would yield the corresponding known Finsler structure.
To this end, we first characterize the geodesics and the geodesic distance on $\scriptUA$. 

\begin{proposition}\label{prop:short_curves_script_U}
Let $\| \cdot \|_{\Phi}$ be any smooth unitarily invariant norm and consider the Finsler structure given by $F_{\scriptUA}^{\Phi} : T \scriptUA \to \mR_+$, $F_{\scriptUA}^{\Phi} : (U,X) \mapsto \| X \|_{\Phi}$. A geodesic between $U \in \scriptUA$ and $V \in \scriptUA$ is given by
\[
\gamma_{\scriptUA}(t) = Ue^{t \log(U^{-1}V)} = U^{1/2} (U^{-1/2} V U^{-1/2})^t U^{1/2}.
\]
Moreover, the geodesic distance is given by
\begin{equation}\label{eq:dist_manifold_AU}
\delta_{\scriptUA}^{\Phi}(U, V) =  \| \log(U^{-1}V) \|_{\Phi} =  \| \log(U^{-1/2} V U^{-1/2}) \|_{\Phi}.
\end{equation}
\end{proposition}

\begin{proof}
Let $U,V \in \scriptUA$. The proposition follows if we can show that a geodesic on $\mU_n$ between $U$ and $V$, given by \eqref{eq:geodesic_on_U}, remains in $\scriptUA$ for $t\in[0,1]$, and that $Z$ in \eqref{eq:geodesic_on_U} and \eqref{eq:geodesic_distance_U} takes the form $Z = -i\log(U^*V)$. To show the latter, 
note that
$e^{iZ} = U^{-1}V = U^*V$,
where $\lambda(U^*V) \cap \mR_- = \emptyset$ since both $U^*$ and $V$ are strictly accretive, see \cite{thompson1974eigenvalues}, \cite[Thm.~6.2]{wang2019phases}.
Therefore we can use the principle branch of the logarithm, which gives
$Z = -i\log(U^*V)$.
Next, to show that $\gamma_{\mU_n}(t) \in \scriptUA$ for $t \in [0,1]$, note that $\gamma_{\mU_n}(1/2) = U^{1/2} (U^{-1/2} V U^{-1/2})^{1/2} U^{1/2}$. By \cite[Prop.~3.1 and Thm.~3.4]{drury2015principal},  $\gamma_{\mU_n}(1/2)$ is strictly accretive, and a repeated argument now gives that $\gamma_{\mU_n}(t)$ is strictly accretive for 
a dense set of $t \in [0,1]$. By continuity of the map $t \mapsto \gamma_{\mU_n}(t)$, the result follows.
\end{proof}

Next, let $\Phi_1$ and $\Phi_2$ be two smooth symmetric gauge functions and consider the Finsler manifolds $(\mP_n, F_{\mP_n}^{\Phi_1})$ and $(\scriptUA, F_{\mU_n}^{\Phi_2})$ as defined in Proposition~\ref{prop:new_form_curve_P} and Proposition~\ref{prop:short_curves_script_U}, respectively. 
In the Riemannian case, there is a canonical way to introduce a metric on $\mP_n \times \scriptUA$, namely the product metric \cite[Ex.~13.2]{lee2013introduction}.
However, in the case of products of Finsler manifolds there is no canonical way to introduce a Finsler structure on a product space, cf.~\cite[Ex.~11.1.6]{bao2000introduction}, \cite{okada1982minkowskian}.
Here, we consider so called (Minkowskian) product manifolds \cite{okada1982minkowskian} and 
to this end we next define so called (Minkowskian) \emph{product functions}.

\begin{definition}[\cite{okada1982minkowskian}]\label{def:prod_fun}
A function $\Psi : \mR_+ \times \mR_+ \to \mR_+$ is called a \emph{product function} if it satisfies the following conditions:
\begin{enumerate}[i)]
\item $\Psi(x_1, x_2) = 0$ if and only if $(x_1, x_2) = (0, 0)$,
\item $\Psi(\alpha x_1, \alpha x_2) = \alpha \Psi(x_1, x_2) $ for all $(x_1, x_2) \in \mR_+ \times \mR_+$ and all $\alpha \in \mR_+$,
\item $\Psi$ is smooth on $\mR_+ \times \mR_+ \setminus \{ (0,0) \}$, 
\item $\partial_{x_\ell} \Psi \neq 0$  on $\mR_+ \times \mR_+ \setminus \{ (0,0) \}$, for $\ell = 1, 2$,
\item $\partial_{x_1} \Psi \, \partial_{x_2} \Psi - 2 \Psi \partial_{x_1}\partial_{x_2} \Psi \neq 0$ on $\mR_+ \times \mR_+ \setminus \{ (0,0) \}$.
\end{enumerate}
\end{definition}

For any product function $\Psi$,  $(\mP_n \times \scriptUA, \sqrt{\Psi((F_{\mP_n}^{\Phi_1})^2, ( F_{\scriptUA}^{\Phi_2})^2)})$ is a Finsler manifold \cite{okada1982minkowskian},
and we therefore define the Finsler manifold $(\mA_n, F_{\mA_n}^{\Phi_1, \Phi_2, \Psi})$ as follows.%
\footnote{In \cite{okada1982minkowskian}, the convention is that the Finsler structure is squared compared to the one in \cite{bao2000introduction}.
We follow the convention of the latter.}%
\footnote{Note that functions $\Psi$ fulfilling i)-v) are not necessarily symmetric gauge functions. As an example, consider $\Psi(x,y) = (x^2 + 3xy + y^2)^2$ \cite[Rem.~6]{okada1982minkowskian}; this is not a symmetric gauge function since in general $\Psi(x, y) \neq \Psi(x, -y)$. Conversely,  symmetric gauge functions do not in general fulfill i)-v). As an example, consider $\Phi(x,y) = \max\{ |x|, |y| \}$ \cite[p.~138]{marshall2011inequalities}; at any point $(x,y) \in \mR_+ \times \mR_+$ such that $x > y$, $\partial_y \Phi = 0$ and hence condition iv) is not fulfilled for this function.}

\begin{definition}
Let $\Phi_1$ and $\Phi_2$ be two smooth symmetric gauge functions, 
let $(\mP_n, F_{\mP_n}^{\Phi_1})$ and $(\scriptUA, F_{\mU_n}^{\Phi_2})$ be the Finsler manifolds as defined in Proposition~\ref{prop:new_form_curve_P} and Proposition~\ref{prop:short_curves_script_U}, respectively, and let $\Psi$ be a product function. The Finsler manifold $(\mA_n, F_{\mA_n}^{\Phi_1, \Phi_2, \Psi})$ is defined via the diffeomorphis in Proposition~\ref{prop:smooth_varying_SPD} as
\[
(\mA_n, F_{\mA_n}^{\Phi_1, \Phi_2, \Psi}) := \left(\mP_n \times \scriptUA, \sqrt{\Psi((F_{\mP_n}^{\Phi_1})^2, ( F_{\scriptUA}^{\Phi_2})^2}) \right).
\] 
\end{definition}

One particular example of a product function is $\Psi : (x_1, x_2) \mapsto x_1 +x_2$, which in \cite{okada1982minkowskian} this was called ``the Euclidean product'', and in the Riemannian case this leads to the canonical product manifold.
Moreover, the geodesics and geodesic distance can be characterized for general product functions $\Psi$. This leads to the following result.

\begin{theorem}\label{thm:main_result}
Let $A, B \in \mA_n$, and let $A = P_A U_A P_A$ and $B = P_B U_B P_B$ be the corresponding symmetric polar decompositions. On the Finsler manifold $(\mA_n, F_{\mA_n}^{\Phi_1, \Phi_2, \Psi})$, a geodesic from $A$ to $B$ is given by
\begin{subequations}\label{eq:geodesics_on_A}
\begin{equation}\label{eq:geodesics_on_A_1}
\gamma_{\mA_n}(t) = \gamma_{\mP_n}(t)^{1/2} \cdot \gamma_{\scriptUA}(t) \cdot \gamma_{\mP_n}(t)^{1/2},
\end{equation}
where
\begin{align}
& \gamma_{\mP_n}(t) := P_A(P_A^{-1} P_BP_B P_A^{-1})^t P_A = P_A^2 e^{t \log(P_A^{-2} P_B^2)}, \label{eq:geodesics_on_A_2} \\
& \gamma_{\scriptUA}(t) := U_A^{1/2} (U_A^{-1/2} U_B U_A^{-1/2})^t U_A^{1/2} = U_A e^{t \log(U_A^*U_B)}. \label{eq:geodesics_on_A_3}
\end{align}
\end{subequations}
Moreover, the geodesic distance from $A$ to $B$ is given by
\begin{align}
\delta_{\mA_n}^{\Phi_1, \Phi_2, \Psi}(A, B) & = \sqrt{\Psi\Big(\delta_{\mP_n}^{\Phi_1}(P_AP_A, P_BP_B)^2, \delta_{\scriptUA}^{\Phi_2}(U_A, U_B)^2\Big)} \nonumber \\
& = \sqrt{\Psi\Big(\|\log (P_A^{-1} P_BP_B P_A^{-1} )\|_{\Phi_1}^2, \|\log(U_A^*U_B) \|_{\Phi_2}^2\Big)} \label{eq:distance_on_A} \\
& = \sqrt{\Psi\Big(\Phi_1(\lambda( \log(P_A^{-1} P_BP_B P_A^{-1} )))^2, \Phi_{2}(\varphi(U^{-1}_AU_B ))^2\Big)}, \nonumber
\end{align}
where $\varphi(\cdot)$ denotes the angles of the eigenvalues.
\end{theorem}

\begin{proof}
That $(\mA_n, F_{\mA_n}^{\Phi_1, \Phi_2, \Psi})$ is a Finsler manifold follows from the discussion leading up to the theorem; see \cite{okada1982minkowskian}.
Moreover, that \eqref{eq:geodesics_on_A} is a geodesic follows (by construction) by using \cite[Thm.~3]{okada1982minkowskian} together with  Proposition~\ref{prop:new_form_curve_P} and Proposition~\ref{prop:short_curves_script_U}; this also gives the first two equalities in \eqref{eq:distance_on_A}. To prove the last equality, first observe that $P^{-1}P_BP_BP^{-1}$ is the geometric mean of $P_A^2$ and $P_B^2$ and hence positive definite \cite[Thm.~4.1.3]{bhatia2007positive}, \cite[Sec.~3]{drury2015principal}. Therefore, $\log (P_A^{-1} P_BP_B P_A^{-1})$ is Hermitian and thus 
\[
\sigma(\log (P_A^{-1} P_BP_B P_A^{-1}) = |\lambda(\log (P_A^{-1} P_BP_B P_A^{-1}))|.
\]
Similarly, $U^{-1}_AU_B$ is unitary and thus $\log(U_A^{-1}U_B)$ is skew-Hermitian. Therefore, $\lambda(\log(U_A^{-1}U_B)) = -i\varphi(U_A^{-1}U_B)$ and hence 
\[
\sigma(\log(U_A^{-1}U_B)) = | \lambda(\log(U_A^{-1}U_B)) | = |\varphi(U_A^{-1}U_B)|.
\] 
Finally, for unitary invariant norms we have that $\| \cdot \|_{\Phi} = \Phi(\sigma(\cdot))$, and since
for any symmetric gauge function $\Phi(|x|) = \Phi(x)$, the last equality follows.
\end{proof}

Next, we derive some properties of the geodesic distance in Theorem~\ref{thm:main_result}.

\begin{proposition}\label{prop:properties_of_geo_dist}
For matrices $A,B \in (\mA_n, F_{\mA_n}^{\Phi_1, \Phi_2, \Psi})$ we have that
\begin{enumerate}[1)]
\item 
$ \displaystyle
\delta_{\mA_n}^{\Phi_1, \Phi_2, \Psi}(A^{-1}, B^{-1}) = \delta_{\mA_n}^{\Phi_1, \Phi_2, \Psi}(A, B) 
$
\item 
$ \displaystyle
\delta_{\mA_n}^{\Phi_1, \Phi_2, \Psi}(A^*, B^*) = \delta_{\mA_n}^{\Phi_1, \Phi_2, \Psi}(A, B)
$
\item 
$ \displaystyle
\delta_{\mA_n}^{\Phi_1, \Phi_2, \Psi}(A^{-1}, A) = 2\delta_{\mA_n}^{\Phi_1, \Phi_2, \Psi}(I, A) = 2\delta_{\mA_n}^{\Phi_1, \Phi_2, \Psi}(I, A^{-1})$,
and the geodesic midpoint between $A^{-1}$ and $A$ is $\gamma_{\mA_n}(1/2) = I$
\item  for any $U \in \mU_n$ we have that
$ \displaystyle
\delta_{\mA_n}^{\Phi_1, \Phi_2, \Psi}(U^*AU, U^*BU)  = \delta_{\mA_n}^{\Phi_1, \Phi_2, \Psi}(A, B).
$
\end{enumerate}
\end{proposition}

\begin{proof}
To prove the statements, first note that $A^{-1} = P_A^{-1}U_A^{-1}P_A^{-1}$, and that $A^* = P_A U_A^{-1} P_A$.
To prove 1), we observe that
\[
\delta_{\mA_n}^{\Phi_1, \Phi_2, \Psi}(A^{-1}, B^{-1}) = \sqrt{\Psi\Big(\Phi_1(\lambda(\log(P_A P_B^{-1}P_B^{-1} P_A)))^2, \Phi_{2}(\varphi(U_AU_B^{-1} ))^2\Big)}.
\]
For the positive definite part, we have that 
\[
\lambda(\log(P_A P_B^{-1}P_B^{-1} P_A)) = \lambda(-\log((P_A P_B^{-1}P_B^{-1} P_A)^{-1})) = -\lambda(\log(P_A^{-1} P_BP_B P_A^{-1})),
\]
and since the symmetric gauge function $\Phi_1$ is invariant under sign changes the distance corresponding to the positive definite part is equal.
Similarly, for the strictly accretive unitary part we have that $|\varphi(U_AU_B^{-1})| = |\varphi((U_AU_B^{-1})^*)| = |\varphi(U_BU_A^{-1})| = |\varphi(U_A^{-1}U_B)|$, where the first equality follows from that the absolute value of the angles of the eigenvalues of a unitary matrix are invariant under the operation of taking conjugate transpose, and the last equality follows since the angles of the eigenvalues are invariant under unitary congruence.
Statement 2) follows by a similar argument. To prove statement 3), 
\begin{align*}
\delta_{\mA_n}^{\Phi_1, \Phi_2, \Psi}(A^{-1}, A) &= \sqrt{\Psi\Big(\Phi_1(\lambda(\log(P_A^{-4})))^2, \Phi_{2}(\varphi(U_A^{-2} ))^2\Big)} \\
& = \sqrt{\Psi\Big(\Phi_1(-2 \lambda(\log(P_A^{2})))^2, \Phi_{2}(-2\varphi(U_A ))^2\Big)} \\
& = 2\sqrt{\Psi\Big(\Phi_1(\lambda(\log(P_A^{2})))^2, \Phi_{2}(\varphi(U_A ))^2\Big)} = 2\delta_{\mA_n}^{\Phi_1, \Phi_2, \Psi}(I, A)
\end{align*}
where the second equity follows by an argument similar to previous ones, and the third equality follows from property ii) for symmetric gauge functions and property ii) for product functions. The second equality in 3) now follows from 1), and that $\gamma_{\mA_n}(1/2) = I$ follows by a direct calculation using \eqref{eq:geodesics_on_A}. Finally, to prove 4), simply note that $U^*AU = U^*P_AU_AP_AU = U^*P_AU U^* U_A U U^*P_AU$, i.e., the same unitary congruence transformation applied to $P_A$ and $U_A$ individually. A direct calculation, using the unitary invariance of eigenvalues, the matrix logarithm, and the norms, gives the result. 
\end{proof}

\begin{remark}
Note that $(\mA_n, F_{\mA_n}^{\Phi_1, \Phi_2, \Psi})$ is in general not
a complete metric space. In particular, in the Riemannian case, i.e., with symmetric gauge functions $\Phi_\ell : x \in \mR^n \mapsto \sqrt{\sum_{k = 1}^n x_k^2}$, $\ell = 1, 2$, and product function $\Psi : (x_1 ,x_2) \mapsto x_1 + x_2$, 
$(\mA_n, F_{\mA_n}^{\Phi_1, \Phi_2, \Psi})$
is not a complete metric space since it is not geodesically complete \cite[Thm.~6.19]{lee2018introduction}.
The latter is due to the fact that $(\scriptUA, F_{\scriptUA}^{\Phi_2})$ is not geodesically complete; for $(\scriptUA, F_{\scriptUA}^{\Phi_2})$ geodesics are not defined for all $t \in \mR_+$ since they will reach the boundary. For an example of a Cauchy sequence that does not converge to an element in $(\mA_n, F_{\mA_n}^{\Phi_1, \Phi_2, \Psi})$, consider the sequence $(A_\ell)_{\ell = 1}^\infty$, where $A_\ell = e^{i(\pi/2 - \pi/(2\ell))}I \in \mA_n$ for all $\ell$. In this case, the geodesic distance between $A_\ell$ and $ A_{k}$ is given by
\[
\delta_{\mA_n}(A_\ell, A_{k}) =  \| Z^{(\ell, k)} \| = \sqrt{n} \frac{\pi}{2} \left| \frac{1}{\ell} - \frac{1}{k} \right|,
\]
since $Z^{(\ell, k)} := -i\log(A_\ell^* A_k) = -i\log(e^{i(\pi/(2\ell) - \pi/(2k))}I) = \pi/2 (1/\ell - 1/k)I$. Thus  $(A_\ell)_{\ell = 1}^\infty$ is a Cauchy sequence, however
$\lim_{\ell \to \infty} A_\ell = e^{i\pi/2}I \not \in \mA_n$.
Since the Hopf-Rinow theorem \cite[Thm.~6.19]{lee2018introduction} also carries over to Finsler geometry \cite[Thm.~6.6.1]{bao2000introduction}, similar statements are true also in the general case.
\end{remark}

\begin{remark}
Using Remark~\ref{rem:rotation_of_A}, the above results can easily be generalized to the same subsets $\tilde{\mA}_n$ of sectorial matrices.
In fact, a direct calculation shows that all the algebraic expressions in Theorem~\ref{thm:main_result} still hold in this case.
However, statements 1)-3) of Proposition~\ref{prop:properties_of_geo_dist} use the fact that if $A \in \mA_n$ then $A^{-1}, A^* \in \mA_n$. This is in general not true for other sets $\tilde{\mA}_n$.
\end{remark}

\section{An application and some related results}\label{sec:applications}

By construction, on the Finsler manifold $(\mA_n, F_{\mA_n}^{\Phi_1, \Phi_2, \Psi})$ the question ``given $A \in \mA_n$, which matrix $B \in \mP_n$ is closest to $A$'' have the answer ``$B = P^2$, where $A = PUP$ is the symmetric polar decomposition.'' Similarly, the corresponding question ``which matrix $B \in \scriptUA$ is closest to $A$'' have the answer ``$B = U$''.
In this section we consider an application of the distance on the Finsler manifold $(\mA_n, F_{\mA_n}^{\Phi_1, \Phi_2, \Psi})$ to another matrix approximation problem, namely finding the closest matrix of bounded log-rank, the definition of which is given in Section~\ref{subsec:logrank}.
Moreover, in Section~\ref{sec:geometric_mean} we consider the relation between the midpoint of geodesics on $(\mA_n, F_{\mA_n}^{\Phi_1, \Phi_2, \Psi})$ and the geometric mean of strictly accretive matrices as introduced in \cite{drury2015principal}.

\subsection{Closest matrix of bounded log-rank}\label{subsec:logrank}

For a positive definite matrix $P$, we defined the log-rank as the rank of the matrix logarithm of $P$. This is equivalent to the number of eigenvalues of $P$ that are different from $1$. We denote this $\logrank_{\mP_n}(\cdot)$. Analogously, for a unitary matrix $U$ the log-rank can be defined as the rank of the matrix logarithm of $U$, which is equivalent to the number of eigenvalues of $U$ with phase different from $0$. We denote this $\logrank_{\mU_n}(\cdot)$.
For strictly accretive matrices we define the log-rank as follows.

\begin{definition}\label{def:log-rank}
For $A \in \mA_n$ with symmetric polar decomposition $A = PUP$, we define the log-rank of $A$ as
\[
\logrank_{\mA_n}(A) := \max\{ \logrank_{\mP_n}(P^2),\; \logrank_{\mU_n}(U) \}.
\]
\end{definition}

We now consider the log-rank approximation problem: given $A \in \mA_n$ find $A_r \in \mA_n$, the latter with log-rank bounded by $r$, that is closest to $A$ in the geodesic distance $\delta_{\mA_n}^{\Phi_1, \Phi_2, \Psi}$. This can be formulated as the optimization problem  
\begin{subequations}\label{eq:log-rank}
\begin{align}
\inf_{A_r \in \mA_n} & \quad \delta_{\mA_n}^{\Phi_1, \Phi_2, \Psi}(A_r, A) \\
\text{subject to} & \quad \logrank_{\mA_n}(A_r) \leq r.
\end{align}
\end{subequations}

Let $A_r = P_r U_r P_r$ be the symmetric polar decomposition.
By properties i) - iv) in the Definition~\ref{def:prod_fun} of product functions $\Psi$, for each such function the distance is nondecreasing in each argument separately.
Therefore, by the form of the geodesic distance \eqref{eq:distance_on_A} and the definition of log-rank on $\mA_n$, it follows that \eqref{eq:log-rank} splits into two separate problems over $\mP_n$ and $\scriptUA$, namely
\begin{subequations}\label{eq:log-rank_P}
\begin{align}
\inf_{P_r \in \mP_n} & \quad \|\log (P_r^{-1} P^2 P_r^{-1} )\|_{\Phi_1} \label{eq:log-rank_P_cost} \\
\text{subject to} & \quad \logrank_{\mP_n}(P_r^2) \leq r,
\end{align}
\end{subequations}
and
\begin{subequations}\label{eq:log-rank_U}
\begin{align}
\inf_{U_r \in \scriptUA} & \quad \|\log(U_r^*U) \|_{\Phi_2} \\
\text{subject to} & \quad \logrank_{\mU_n}(U_r) \leq r.
\end{align}
\end{subequations}
In fact, this gives the following theorem.

\begin{theorem}
Assume that $\hat{P}_r$ and $\hat{U}_r$ are optimal solutions to \eqref{eq:log-rank_P} and \eqref{eq:log-rank_U}, respectively. Then an optimal solution to \eqref{eq:log-rank} is given by $\hat{A_r} = \hat{P}_r \hat{U}_r \hat{P}_r$. Conversely, if \eqref{eq:log-rank} has an optimal solution $\hat{A_r} = \hat{P}_r \hat{U}_r \hat{P}_r$, then $\hat{P}_r$ and $\hat{U}_r$ are optimal solutions to \eqref{eq:log-rank_P} and \eqref{eq:log-rank_U}, respectively.
\end{theorem}

In \cite[Thm.~3]{zhao2020low} it was shown that \eqref{eq:log-rank_U} always has an optimal solution, and that it
is the same for all symmetric gauge functions $\Phi_2$.
More precisely, the optimal solution
$\hat{U}_r$ is obtained from $U$ by setting the $n-r$ phases of $U$ with smallest absolute value equal to $0$. That is, let $U = V^*DV$ be a diagonalization of $U$ where $D = \diag([e^{i\tilde{\phi}_k(U)}]_{k = 1}^n)$ and where $[\tilde{\phi}_k(U)]_{k = 1}^n$ are the phases of $U$ ordered so that $|\tilde{\phi}_1(U)| \geq |\tilde{\phi}_2(U)| \geq \ldots \geq |\tilde{\phi}_n(U)|$. Then 
\[
\hat{U}_r = V^* \diag(e^{i\tilde{\phi}_1(U)}, \ldots e^{i\tilde{\phi}_r(U)}, 1, \ldots, 1) V
\]
is the optimal solution to \eqref{eq:log-rank_U}.
In the same spirit, the optimal solution to \eqref{eq:log-rank_P} can be charaterized as follows.

\begin{proposition}
Let $P \in \mP_n$, and let $P^2 = V^* \diag(\lambda_1, \ldots, \lambda_n) V$ be a diagonlization of $P^2$ where the eigenvalues are ordered so that $|\log(\lambda_1)| \geq |\log(\lambda_2)| \geq \ldots \geq |\log(\lambda_n)|$.
Then $\hat{P}_r^2 = V^* \diag(\lambda_1, \ldots, \lambda_r, 1, \ldots, 1) V$
is a minimizer to \eqref{eq:log-rank_P} for all symmetric gauge functions $\Phi_1$.
\end{proposition}

\begin{proof}
Clearly, $\logrank_{\mP_n}(\hat{P}_r) = r$ and hence $\hat{P}_r$ is feasible to \eqref{eq:log-rank_P}. Next, $\|\log (P_r^{-1} P^2 P_r^{-1} )\|_{\Phi_1} = \Phi_1(\lambda(\log(P_r^{-1}P^2P_r^{-1})))$. Moreover, since $P_r^{-1}P^2P_r^{-1} \in \mP_n$ and hence is unitary diagonalizable and has positive eigenvalues, we have that $\lambda(\log(P_r^{-1}P^2P_r^{-1})) = \log(\lambda(P_r^{-1}P^2P_r^{-1}))$.
Now, 
to show that $\hat{P}_r$ is the minimizer to \eqref{eq:log-rank_P} for all symmetric gauge functions $\Phi_1$, it is equivalent to show that $|\log(\lambda(\hat{P}_r^{-1} P^2 \hat{P}_r^{-1}))|  \prec_w |\log(\lambda(P_r^{-1} P^2 P_r^{-1}))|$ for all $P_r$ such that $\logrank_{\mP_n}(P_r^2) \leq r$; see, e.g., \cite[Thm.~4]{fan1951maximum}, \cite[Thm.~1]{mirsky1960symmetric}, \cite[Sec.~3.5]{horn1994topics}, \cite[Prop.~4.B.6]{marshall2011inequalities}, \cite[Thm.~10.35 ]{zhang2011matrix}. 

To this end, using \cite[Thm.~9.H.1.f]{marshall2011inequalities} (or \cite[Cor~III.4.6 ]{bhatia1997matrix}, \cite[Thm.~10.30 ]{zhang2011matrix}) we have that $\log(\lambda(P^2))^\downarrow - \log(\lambda(P_r^2))^\downarrow \prec \log(\lambda(P_r^{-1}P^2P_r^{-1}))$,%
\footnote{To see this, take $V = P_r^{-1}P^2P_r^{-1} \in \mP_n$ and $U = P_r^2$ in \cite[Thm.~9.H.1.f]{marshall2011inequalities}, and use the fact that the eigenvalues of $UV = P_rP^2P_r^{-1}$ are invariant under the similarity transform $P_r^{-1} \cdot P_r$.}
and by \cite[Ex~11.3.5]{bhatia1997matrix} we therefore have that $|\log(\lambda(P^2))^\downarrow - \log(\lambda(P_r^2))^\downarrow| \prec_w |\log(\lambda(P_r^{-1}P^2P_r^{-1}))|$.
Moreover,
by a direct calculation it can be verified that 
$\log(\lambda(P^2))^\downarrow - \log(\lambda(\hat{P}_r^2))^\downarrow = \log(\lambda(\hat{P}_r^{-1} P^2 \hat{P}_r^{-1}))^\downarrow$ holds.
Therefore, if we can show that
\begin{equation}\label{eq:proof_eq1}
\begin{aligned}
\text{$|\log(\lambda(P^2))^\downarrow - \log(\lambda(\hat{P}_r^{2}))^\downarrow| \prec_w |\log(\lambda(P^2))^\downarrow - \log(\lambda(P_r^{2}))^\downarrow|$}&\\
\text{for all for all $P_r$ such that $\logrank_{\mP_n}(P_r^2) \leq r$} & ,
\end{aligned}
\end{equation}
we would have that for all such $P_r$,
\begin{align*}
|\log(\lambda(\hat{P}_r^{-1} P^2\hat{P}_r^{-1}))^\downarrow| & = |\log(\lambda(P^2))^{\downarrow} - \log(\lambda(\hat{P}_r^{2}))^\downarrow| \prec_w |\log(\lambda(P^2))^\downarrow - \log(\lambda(P_r^{2}))^\downarrow| \\
&  \prec_w |\log(\lambda(P_r^{-1} P^2 P_r^{-1}))|
\end{align*}
and by transitivity of preorders the result follows.
To show \eqref{eq:proof_eq1}, we formulate the following equivalent optimization problem: let $a = \log(\lambda(P^2))^{\downarrow}$
and consider
\begin{align*}
\min_{\prec_w} & \quad  |a - x| \\
\text{subject to} & \quad  x \in \mR^n, \quad x_1 \geq x_2 \geq \ldots \geq x_n\\
& \quad \text{at most } r \text{ elements of } x \text{ are nonzero},
\end{align*}
where $\min_{\prec_w}$ is minimizing with respect to the preordering $\prec_w$.
The solution to the latter is to take $x_i = a_i$ for the $r$ elements of $a$ with largest absolute value.
\end{proof}

\subsection{On the geometric mean for strictly accretive matrices}\label{sec:geometric_mean}
The geometric mean of strictly accretive matrices, denoted by $A\#B$, was introduced in \cite{drury2015principal} as a generalization of the geometric mean for positive definite matrices \cite[Chp.~4 and 6]{bhatia2007positive}, \cite{moakher2005differential}. In particular, in \cite{drury2015principal} it was shown that for $A,B \in \mA_n$ there is a unique solution $G \in \mA_n$ to the equation $GA^{-1}G = B$. This solution is given explicitly as 
\[
A\#B := G = A^{1/2}(A^{-1/2}BA^{-1/2})^{1/2}A^{1/2},
\]
which is also the same algebraic expression as for the geometric mean of positive definite matrices.

The geometric mean for positive definite matrices can also be interpreted as the midpoint on the geodesic connecting 
the matrices
\cite[Sec.~6.1.7]{bhatia2007positive}.
With the Finsler geometry $(\mA_n, F_{\mA_n}^{\Phi_1, \Phi_2, \Psi})$, we can therefore get an alternative definition of the geometric mean between strictly accretive matrices as the geodesic midpoint.
However,
for $A, B \in (\mA_n, F_{\mA_n}^{\Phi_1, \Phi_2, \Psi})$ we
in general have that $\gamma_{\mA_n}(1/2) \neq A \# B$. This can be seen by the following simple example.

\begin{example}
Let $A = I$ and let $B = PUP \in  (\mA_n, F_{\mA_n}^{\Phi_1, \Phi_2, \Psi})$. Then we have that $A \# B = B^{1/2} = (PUP)^{1/2}$ and $\gamma_{\mA_n}(1/2) = P^{1/2}U^{1/2}P^{1/2}$. Thus, in general $A \# B \neq \gamma_{\mA_n}(1/2)$. In fact, equality holds in this case if and only if $P$ and $U$ commute.
\end{example}

Instead, the midpoint of the geodesic between two matrices $A,B \in  (\mA_n, F_{\mA_n}^{\Phi_1, \Phi_2, \Psi})$ can be expressed using the geometric mean $\#$ as 
\begin{equation}\label{eq:geo_mean_and_midpoint}
\gamma_{\mA_n}(1/2) = (P_A^2 \# P_B^2)^{1/2} \, (U_A\#U_B) \, (P_A^2 \# P_B^2)^{1/2},
\end{equation}
which follows directly from \eqref{eq:geodesics_on_A}. Using this representation, we can characterize when $A \# B = \gamma(1/2)$. In order to do so, we first need the following two auxiliary result.

\begin{lemma}\label{lem:normal_A_polar_decomp}
Let $A \in \mD_n$, and let $A = VQ$ be is polar decomposition, where $V \in \mU_n$ and $Q \in \mP_n$. $A$ is normal if and only if $A = Q^{1/2}VQ^{1/2}$ is the symmetric polar decomposition of $A$.
\end{lemma}

\begin{proof}
First, using \cite[Lem.~9]{horn1959eigenvalues} we conclude that since $A$ is sectorial, $V$ is also sectorial. Now, $A$ is normal if and only if $V$ and $Q$ commute \cite[Thm.~9.1]{zhang2011matrix}, which is true if and only if $V$ and $Q^{1/2}$ commute. Hence $A$ is normal if and only if $A = VQ = Q^{1/2}VQ^{1/2}$, and by the existence and uniqueness of the symmetric polar decomposition the result follows.
\end{proof}

\begin{lemma}\label{lem:geo_mean_congruence}
Let $A,B \in \mA_n$ and $G = A\#B$. For all $X \in \gln$, the unique strictly accretive solution to 
\[
H(X^*AX)^{-1}H = X^*BX
\]
is $H = X^{*}GX$.
\end{lemma}

\begin{proof}
That $H = X^{*}GX$ solves the equations is easily verified by simply plugging it in. Moreover, that $H$ is unique follows from the uniqueness of the  geometric mean for strictly accretive matrices \cite[Sec.~3]{drury2015principal} and the fact that for any $X \in \gln$ we have that $X^*AX, X^*BX \in \mA_n$.
\end{proof}

\begin{proposition}\label{prop:geomean_and_midpoint}
For $A, B \in  (\mA_n, F_{\mA_n}^{\Phi_1, \Phi_2, \Psi})$, let $A = P_A U_A P_A$ and $B = P_B U_B P_B$ be the corresponding symmetric polar decompositions. We have that $A \# B = \gamma_{\mA_n}(1/2)$
if one of the following holds:
\begin{enumerate}[i)]
\item $U_A = U_B = I$,
\item $P_A = P_B$,
\item $A$ and $B$ are commuting normal matrices.
\end{enumerate}
\end{proposition}

\begin{proof}
Using \eqref{eq:geo_mean_and_midpoint}, the first statement follows immediately.
To prove the second statement, let $A = PU_AP$ and $B = PU_BP$. From \eqref{eq:geo_mean_and_midpoint} it therefor follows that $\gamma_{\mA_n}(1/2) = P (U_A\#U_B) P$. Using Lemma~\ref{lem:geo_mean_congruence} with $G = U_A \# U_B$ and $X = P$, where therefore have that
\[
A\#B = (P U_A P)\# (P U_B P) = P (U_A\#U_B) P = \gamma_{\mA_n}(1/2).
\]

To prove the third statement, by Lemma~\ref{lem:normal_A_polar_decomp} we have that $P_A$, $U_A$, and $P_B$, $U_B$ commute. Moreover, since commuting normal matrices are simultaneously unitarilty diagonalizable \cite[Thm.~2.5.5]{horn2013matrix}, and since a unitary diagonlization is unique up to permutation of the eigenvalues and eigenvectors, it follows that $P_A, U_A, P_B$ and $U_B$ all commute. Using this together with \eqref{eq:geo_mean_and_midpoint}, a  direct calculation gives the result.
\end{proof}

As noted in the above proof, if $A$ and $B$ are normal and commute they are also simultaneously unitarily diagonalizable \cite[Thm.~2.5.5]{horn2013matrix}, i.e., $A = V^* \Lambda_A V$ and $B = V^* \Lambda_B V$ for some $V \in \mU_n$. In this case, using Lemma~\ref{lem:geo_mean_congruence} we have that $A \# B = V^* (\Lambda_A \# \Lambda_B) V$, and the geometric mean between $A$ and $B$ can thus be interpreted as the (independent) geometric mean between the corresponding pairs of eigenvalues.
In fact, the latter observation can be generalized to all pairs of matrices that can be simultaneously diagonalized by congruence, albeit that the elements of the diagonal matrices are not necessarily eigenvalues in this case (cf.~Proposition~\ref{prop:new_form_curve_P}).

\begin{proposition}\label{prop:geodesics_and_geometric_mean}
Let $A,B \in \mA_n$ and assume that $A = T^* D_A T$ and $B = T^* D_B T$, where $T \in \gln$ and where $D_A$ and $D_B$ are diagonal matrices. 
Then $A \# B = T^* (D_A \# D_B) T$.
\end{proposition}

\begin{proof}
Let $A, B \in \mA_n$ and assume that $A = T^*D_AT$ and $B = T^* D_B T$, where $T \in \gln$ and where $D_A$ and $D_B$ are diagonal matrices. A direct application of Lemma~\ref{lem:geo_mean_congruence}, with $G = D_A \# D_B$ and $X = T$, gives the result.
\end{proof}

As a final remark, note that if $D_A$ and $D_B$ in Proposition~\ref{prop:geodesics_and_geometric_mean} are unitary, then $A = T^* D_A T$ and $B = T^* D_B T$ are sectorial decompositions of $A$ and $B$, respectively.
Now, let $T = VP$ be the polar decomposition of $T$, with $V \in \mU_n$ and $P \in \mP_n$.
Hence we have that $A = PV^*D_AVP = PU_AP$ and $B = PV^*D_BVP = PU_BP$, i.e., $P$ is the positive definite part and $V^*D_AV$ and $V^*D_BV$ are the strictly accretive unitary part in the symmetric polar decomposition of $A$ and $B$, respectively. From Proposition~\ref{prop:geomean_and_midpoint}.ii) we therefore have that $A \# B = \gamma_{\mA_n}(1/2)$ in this case.

\section{Conclusions}
In this work we show that the set of strictly accretive matrices is a smooth manifold that is diffeomorphic to a direct product of the smooth manifold of positive definite matrices and the smooth manifold of strictly accretive unitary matrices. Using this decomposition, we introduced a family of Finsler metrics and studied their geodesics and geodesic distances. Finally, we consider the matrix approximation problem of finding the closest strictly accretive matrix of bounded log-rank, and also discuss the relation between the geodesic midpoint and the previously introduced geometric mean between accretive matrices.

There are several interesting ways in which these results can be extended. For example, in the case of positive definite matrices the geometric framework offered by the Riemannian manifold construction gives yet another interpretation of the geometric mean. In fact, the geometric mean between two positive definite matrices $A$ and $B$ is also the (unique) solution to the variational problem $\min_{G \in \mP_n} \delta_{\mP_n}(A, G)^2 + \delta_{\mP_n}(B,G)^2$ \cite[Sec.~6.2.8]{bhatia2007positive}, \cite[Prop.~3.5]{moakher2005differential},  and this interpretation can be used to extend the geometric mean to a mean between several matrices \cite{moakher2005differential}. In a similar way, a geometric mean between the strictly accretive matrices $A_1, \ldots, A_N$ can be defined as the solution to
\[
\min_{G \in \mA_n} \; \sum_{i = 1}^N  \delta_{\mA_n}^{\Phi_1, \Phi_2, \Psi}(A_i, G)^2,
\]
however such a generalization would need more investigation. For example, even in the case of the Riemannian metric on the manifold of positive definite matrices, analytically computing the geometric mean between several matrices is nontrivial \cite[Prop.~3.4]{moakher2005differential}. Nevertheless, there are efficient numerical algorithms for solving the latter problem,
see, e.g., the survey \cite{jeuris2012survey} or the monograph \cite{absil2008optimization} and references therein.

The idea of this work was to introduce a metric that separates the ``magnitudes'' and the ``phases'' of strictly accretive matrices.
However, the similarities between the manifold of positive definite matrices and the manifold of unitary matrices raises a question about another potential geometry on $\mA_n$ that does not explicitly use the product structure. More precisely, note that since all strictly accretive matrices have a unique, strictly accretive square root, the inner product on the tangent space $T_U\scriptUA$, given by \eqref{eq:inner_product_U}, can be defined analogously to the one on $\mP_n$, given by \eqref{eq:inner_product_P}, namely as
\[
\langle X,Y \rangle_{U} = \tr((U^{-1/2} X U^{-1/2})^*(U^{-1/2} Y U^{-1/2})) = \tr(X^*Y),
\]
for $U \in \scriptUA$ and $X,Y \in T_U\scriptUA$.
Based on the similarities between the inner products, and the corresponding geodesics and geodesic distances, we ask the following question: for $A, B \in \mA_n$ and $X, Y \in T_A \mA_n$, if we define the inner product on
$T_A\mA_n$ as 
\[
\langle X, Y \rangle_{A} = \tr((A^{-1/2} X A^{-1/2})^*(A^{-1/2} Y A^{-1/2})),
\]
what is the form of the geodesics and the geodesic distance?

\section*{Acknowledgments}
The authors would like to thank Wei~Chen, Dan~Wang, Xin~Mao, Di~Zhao, and Chao~Chen for valuable discussions.

\appendix

\begin{appendices}
\section{Technical results from Section~\ref{sec:A_smooth_manifold}}\label{app:technical_results_one}
The following is a number of lemmata use in the proofs of Theorem~\ref{thm:A_smooth_manifold} and \ref{thm:decomp_manifold}.

\begin{lemma}\label{lem:open_A}
$\mA_n$ is an open sets in $\gln$.
\end{lemma}

\begin{proof}
To show that $\mA_n$ is open in $\gln$, 
note that since $\mH_n \perp \mS_n$, cf.~\cite[Thm.~10.B.1 and 10.B.2]{marshall2011inequalities}, \cite[Prob.~10.7.20]{zhang2011matrix}, we have that $A + B = H(A + B) + S(A + B)$. Moreover, $H(A + B) = H(A) + H(B)$, and since the set $\mP_n$ is open in $\mH_n$, $\mP_n \oplus \mS_n$ is open in $\gln$.
\end{proof}

\begin{lemma}\label{lem:connected}
$\mA_n$is connected.
\end{lemma}

\begin{proof}
By \cite[Prop.~1.11]{lee2013introduction}, $\mA_n$ is connected if and only if it is path-connected. To show the latter, it suffices to show that any $A \in \mA_n$ is path-connected to $I$. To this end, let $A = T^* D T$ be the sectorial decomposition of $A$.
A piece-wise smooth path connecting $A$ and $I$ is given by
\[
\gamma(t) := 
\begin{cases}
T^* D^{1-2t} T, & \text{for } t \in [0, 1/2) \\
(T^* T)^{2 - 2t}, & \text{for } t \in [1/2, 1) \\
I, & \text{for } t = 1,
\end{cases}
\]
and hence $\mA_n$ is connected.
\end{proof}

\begin{lemma}\label{lem:tangent_space}
The tangent space at an $A \in \mA_n$ is given by $T_A\mA_n = \mM_n$.
\end{lemma}

\begin{proof}
This follows since $\mA_n$ is an open subset of $\mM_n$.
\end{proof}

\begin{lemma}\label{lem:script_U_manif}
$\scriptUA$ is a connected smooth manifold  and at a point $U \in \scriptUA$ the tangent space is $T_U\scriptUA = \mS_n$.
\end{lemma}

\begin{proof}
Since $\mA_n$ is open in $\gln$ (Lemma~\ref{lem:open_A}), $\scriptUA = \mU_n \cap \mA_n$ is open in $\mU_n$ in the relative topology with respect to $\gln$. Thus it is a smooth manifold \cite[Ex.~1.26]{lee2013introduction}. Moreover, the proof of Lemma~\ref{lem:connected} holds, \emph{mutatis mutandis}, showing that it is connected. Finally, since it is open in $\mU_n$, the tangent space at $U \in \scriptUA$ is $T_U\scriptUA = \mS_n$, cf.~\cite[Prob.~8.29]{lee2013introduction}.
\end{proof}

\begin{lemma}[{Cf. \cite[Prop.~VII.2.5]{lang1999fundamentals}}]\label{lem:polar_decomp_smooth}
For $A \in \gln$, let $A = VQ$ where $V \in \mU_n$ and $Q \in \mP_n$ be the polar decomposition of $A$. The mapping $A \mapsto (V,Q)$ is a diffeomorphis between the manifolds $\gln$ and $ \mU_n \times \mP_n$.
\end{lemma}

\begin{proof}
First, since $\mU_n$ and $\mP_n$ are smooth manifolds so is $\mU_n \times \mP_n$ \cite[Ex.~1.34]{lee2013introduction}. 
Next, for each $A \in \gln$ the polar decomposition is unique \cite[Thm.~7.3.1]{horn2013matrix}, \cite[Prob.~3.2.20]{zhang2011matrix}, and for each pair of matrices $(V, Q) \in \mU_n \times \mP_n$ we have that $VQ \in \gln$; thus the mapping is bijective. Now, $A$ is smooth in $V$ and $Q$ since it is polynomial in the coefficients, i.e., the inverse mapping is smooth. To prove the converse, note that the components in the polar decomposition $A = VQ$ are given by $Q = (A^*A)^{1/2}$ and $V = A(A^*A)^{-1/2}$, cf.~\cite[p.~449]{horn2013matrix}, \cite[p.~288]{zhang2011matrix}. Since $A \in \gln$, $A^*A \in \mP_n$, and the matrix square root is a smooth function on $\mP_n$, cf.~\cite[Thm.~7.2.6]{horn2013matrix}. Therefore, since both $Q$ and $V$ are compositions of smooth functions of $A$, they both depend smoothly on the components of $A$.
\end{proof}
\end{appendices}

\bibliographystyle{plain}
\bibliography{refs}

\end{document}